\documentclass[11pt]{amsart}

\usepackage{t1enc}
\usepackage{amsthm,amssymb}
\usepackage{latexsym}
\usepackage{amsmath} 
\usepackage{graphics}
\usepackage{epsfig}
\usepackage{color}
\usepackage{amsfonts}
\usepackage[latin1]{inputenc}
\usepackage[english]{babel}
\usepackage{stmaryrd}
\usepackage{hyperref}
\usepackage{xspace}

\newtheorem{theorem}[subsection]{Theorem}
\newtheorem{proposition}[subsection]{Proposition}
\newtheorem{lemma}[subsection]{Lemma}
\newtheorem{corollary}[subsection]{Corollary}
\newtheorem{definition}[subsection]{Definition}
\newtheorem{example}[subsection]{Example}
\newtheorem{remark}[subsection]{Remark}

\newtheorem{notes}[subsection]{Notes}

\newtheorem{problem}{Problem}

\begin{document}

\newcommand{\5}{\vskip 5pt}
\newcommand{\tf}{II$_1$ factor }
\newcommand{\vg}{$vN(\Gamma)$\xspace}
\newcommand{\vgm}{vN_\omega(\Gamma)}
\newcommand{\pslr}{$PSL_2(\mathbb R)$}
\newcommand{\pslz}{$PSL_2(\mathbb Z)$}
\newcommand{\bh}{{\mathcal B}({\mathcal H})}
\newcommand\bk{{\mathcal B}({\mathcal K})}
\newcommand\h{{\mathcal H}}             % careful! these only work inside $ $'s
\newcommand\kh{{\mathcal K}}            % this too!

\newcommand\2{\hspace{2pt}}
\newcommand\1{\hspace{1pt}}
\newcommand\slim{s\hspace{-3pt}-\hspace{-3pt} lim \hspace{2pt}}
\newcommand\wlim{w\hspace{-3pt}-\hspace{-3pt} lim \hspace{2pt}}
\newcommand\vn{von Neumann algebra\xspace}       %% now using the xspace package, to fix occasional typesetting problems
 %% caused by these macros
\newcommand\mt{$M_2(\mathbb C)$}
\newcommand\Dh{\Delta^{1/2} }
\newcommand\bc{\langle M,e_N\rangle}
\newcommand{\Mod}[1]{\ (\mathrm{mod}\ #1)}

\def\hpic #1 #2 {\mbox{$\begin{array}[c]{l} \epsfig{file=#1,height=#2}
\end{array}$}}
 
\def\vpic #1 #2 {\mbox{$\begin{array}[c]{l} \epsfig{file=#1,width=#2}
\end{array}$}}

\def\mypic #1 #2 #3{\mbox{\hspace {#1} $\begin{array}[c]{l} \epsfig{file=#2,width=#3}
\end{array}$}}
\title{ Bergman space zero sets, modular forms, von Neumann algebras\\ and ordered groups.}
\author{Vaughan F. R. Jones}
\thanks{}

\begin{abstract} $A^2_{\alpha}$ will denote the weighted $L^2$ Bergman space.
Given a subset $S$ of the open unit disc we define
% $\omega(S)$ to be $\infty$ or the infimum of $\{s| \exists f \in A^2_{\alpha}, f\neq 0,  \mbox{ with $f$ vanishing on } S\}$ if this set is non-empty, and similarly 
$\Omega(S)$ to be the infimum of $\{s| \exists f \in A^2_{s-2}, f\neq 0,  \mbox{ having $S$ as its zero set} \}$. 
%Clearly $\omega(S)\leq \Omega(S)$ and
 By classical results on Hardy space
there are sets $S$ for which $\Omega(S)=1$. Using von Neumann dimension techniques and cusp forms we give examples of $S$ where $1<\Omega(S)<\infty$. By using a left order on certain Fuchsian groups we are able
to calculate $\Omega(S)$ exactly if $\Omega (S)$ is the orbit of a Fuchsian group. This technique also allows us
to derive in a new way well known results on zeros of cusp forms and indeed calculate the whole algebra of
modular forms for \pslz.

\end{abstract}

\maketitle

\section{Introduction}

The weighted Bergman (Hilbert) spaces $A^2_\alpha$, for $\alpha >-1$ are the spaces of holomorphic functions in the unit disc $\mathbb D$
which are square integrable with respect to the measure $(1-r^2)^\alpha rdrd\theta$ in $\mathbb D$. Obviously 
$A^2_\alpha\subseteq A^2_\beta$ for $\alpha\leq \beta$. The Hardy
space $H^2(\mathbb D)$ (see \cite{zhu}) is contained in all the Bergman spaces. Given a complex valued function $f$ on a set $X$ 
we let $Z_f$ be its zero set,$$Z_f=\{x\in X|f(x)=0\}.$$
We are interested in $Z_f$ for $f\neq 0$ in Bergman space. The case of Hardy space is completely understood  (
see \cite{duren}, apparently by Sz\"ego in about 1915): 
An (obviously countable) subset $\{z_n|n=1,2,\cdots\}$ of $\mathbb D$ is $Z_f$ for some $f\in H^2\setminus\{0\}$ iff $$ \sum_n (1-|z_n|)<\infty.$$ Thus any such set is $Z_f$ for some $f\in A^2_\alpha$ for all $\alpha>-1$ but many other sets
may be zero sets and a similar characterisation for Bergman space seems out of reach.

So for a subset $S\subset \mathbb D$ we let $$X_S=\{s| \exists f \in A^2_{s-2}, f\neq 0, S\subseteq Z_f\} $$%\mbox{     and   }
%Y_S=\{s| \exists f \in A^2_{\alpha}, f\neq 0, S=Z_f\}. $$
Then set $\Omega(S)=\begin{cases} \infty & \mbox{  if } X_S=\emptyset \\
                                                           \inf(X_S)& \mbox{ otherwise}
\end{cases}$.

                 (The shift in values, $s=\alpha+2$ is because for Bergman space the reference measure is 
                 Lebesgue measure in the disc i.e. $\alpha=0$, whereas for cusp forms and von Neumann algebras 
                 the reference measure is hyperbolic area i.e. $s=2$.)                                          
                                            
In \cite{hkz}, $\Omega(S)$ is determined in terms of a notion of ``density'' of points in $S$ which might be difficult to
calculate.  In this paper we 
will show that some progress can be made when $S$ is the orbit of a Fuchsian group $\Gamma$ acting on $\mathbb D$,
effectively calculating the ``density'' of these orbits. 

One can ask whether $\Omega(S)= \inf\{s| \exists f \in A^2_{s-2}, f\neq 0, S=Z_f\}$. Every subset of the zero set of an 
ordinary (unweighted) Bergman space function is the zero set of another function (\cite{horowitz}) but it is
unknown whether this is true for weighted Bergman spaces. However the  result of \cite{hkz} mentioned above makes it 
clear that the only value of $s$ for which it is unknown is  precisely  $\inf\{s| \exists f \in A^2_{s-2}, f\neq 0, S=Z_f\}$ so this is the same as $\Omega(S)$.

Our techniques make particular use of von Neumann algebras as inspired by Atiyah in \cite{a}, both for
the existence and non-existence questions. 
Fuchsian groups act unitarily (projectively) on the Bergman spaces in such a way that they generate what is
known as a II$_1$ factor $M$ (which depends on $s$ and $\Gamma$). 
But these techniques do not so far allow us to get information  on
$Z_f$ itself so for results on $\Omega(S)$ we use modular forms whose zero sets are explicitly known.

Another key idea was proposed by Curt McMullen - that is to exploit the Bergman reproducing kernel 
vectors $\varepsilon_z$ satisfying $\langle \varepsilon_z,f\rangle=f(z)$ for $f$ in a Bergman space. Applying $\Gamma$
(hence $M$) to an $\varepsilon_z$ gives an $M$-module which has a von Neumann dimension which may be
compared to the von Neumann dimension of the Bergman space itself. Standard von Neumann results then lead to
an upper bound for $\Omega(\Gamma(z))$. 
To obtain a lower bound involves showing the existence of a "trace vector" for $M$ and for this we introduce what 
appears to be a new technique.
Every Fuchsian group $\Gamma$ contains a  left orderable subgroup $\Psi$ of finite index \cite{HKS}
($\Gamma$ is either a free product of cyclic 
groups or has a surface group of finite index). To each orbit of a left ordered $\Gamma$ we produce a trace 
vector for \vg which acts on $A^2_{\alpha}$. This produces lower bounds on von Neumann dimension.

By $covolume(\Gamma)$ in this paper we will mean the hyperbolic area (constant curvature $=-1$) of the
quotient space $\mathbb H/\Gamma$.
The simplest version of our main results is the following:

\begin{theorem} \label{ultimate}Let $\Gamma$ be a Fuchsian group and $\Gamma(z)$ be an orbit in $\mathbb H$
containing no fixed points for any element of $\Gamma$. Then
there is a non-zero function in $A^2_\alpha$  vanishing on $\Gamma(z)$ iff $$\alpha>\frac{4\pi}{covolume(\Gamma)}-1$$
\end{theorem}
%\begin{theorem}\label{nobergmanzeroset} Let $\Gamma$ be a Fuchsian group and $\Gamma(z)$ be an orbit in $\mathbb D$ on which $\Gamma$ acts freely. Then\\
%(\romannumeral 1)If $\displaystyle \alpha\leq\frac{4\pi}{covolume(\Gamma)}-1$ there is no non-zero function in $A^2_{\alpha}$ vanishing on $\Gamma(z)$.
%(\romannumeral 2)If $\displaystyle \alpha>\frac{4\pi}{covolume(\Gamma)}-1$ there is a non-zero function in $A^2_{\alpha}$ vanishing on $\Gamma(z)$.
%\end{theorem}

This establishes that $\displaystyle \Omega(\Gamma(z))=\frac{4\pi}{covolume(\Gamma)}+1$ for such groups and since the proof of existence of the appropriate Bergman space functions
 does not require a left ordering we get  $\displaystyle\Omega(\Gamma(z))\leq \frac{4\pi}{covolume(\Gamma)}+1$ for all $\Gamma$.

%In fact we can do away with the fixed point requirement and by manipulating the technique we obtain
%a well known result (at least for \pslz) on the zeros of cusp forms-see of [] Serre. 

%We are actually able to get rid of the orderable assumption as follows but we leave the precise statement of 
%the theorem to \ref{ultimate}:
% Note that every Fuchsian group $\Gamma$ contains a  left orderable subgroup $\Psi$ of finite index (Hoare, Karass, Solitar-either a free product of cyclic 
%groups or has a surface group of finite index). So we may split the orbits of $\Gamma$ up into orbits for $\Psi$.
%The orderability trick works across orbits so provided we pay attention to fixed points for elements of $\Gamma$ we
%obtain the desired result.

The value $\displaystyle\frac{4\pi}{covolume(\Gamma)}+1$ is obviously special and there is a good reason for this.   It is the
value for which the von Neumann dimension $\dim_{\mbox{\vg}}(A^2_{\alpha})$ is equal to $1$ meaning there is a ``cyclic and separating trace vector'' for \vg in the Hilbert space, and hence an antiisomorphism between \vg and its commutant
on $A^2_{\alpha}$. Now in \cite{radulescu2}, Radulescu has shown that the commutant \vg\hspace {-4pt}$'$ is always generated in some sense by 
cusp forms which thus give a model for \vg\hspace {-4pt}$'$. And Voiculescu in \cite{voic} has shown that, at least for groups like \pslz ,
\vg  has a random matrix model. Thus there exists a random matrix model for cusp forms. This is a theorem, but it is of
little use unless one can lay one's hands on an explicit and mangeable cyclic and separating trace vector with
which to implement the anti-isomorphism with the commutant. If  one did have such a vector one
might be able to prove some of the numerically well established relations between random matrices and modular
forms (\cite{keating}). Indeed this was the motivation for the research that led to the results of this paper. 

Interestingly though,  our main
theorem shows that  when $\dim_{\mbox{\vg}}(A^2_{\alpha})=1$, a trace vector, although it exists, can never
be obtained by the left order method of this paper, whereas for all other values of $\alpha$ (for which there
is a trace vector) the left order method works, starting with a cusp form!

We would like to end the introduction by making quite clear what  issue  this paper brings to light.

Fix a Fuchsian group $\Gamma<PSL(2,\mathbb R)$ and let $\gamma\mapsto u_\gamma$ be its 
unitary projective action on $A^2_{s-2}$. An element $\xi\in A^2_{s-2}$ is called \\ \underline{wandering} for
$\Gamma$ iff $$\langle u_\gamma\xi,\xi\rangle=0 \mbox{ for all  }\gamma\in \Gamma, \gamma\neq id,$$ and
\\ \underline{tracelike} if, up to a multiplicative constant, $$  \sum_{\gamma\in\Gamma}\frac{|\xi(z)|^2}{|cz+d|^{2s}}=Im(z)^{-s} $$.
\begin{problem} It follows from von Neumann dimension that if $\displaystyle s=\frac{4\pi}{covolume(\Gamma)}+1$, there is
a nonzero function in $A^2_{s-2}$ that is both wandering and tracelike.\\ Find such a function.
\end{problem}

A first step might be to find a direct (non von Neumann algebraic) proof of the equivalence of the wandering
and tracelike conditions when $\displaystyle s=\frac{4\pi}{covolume(\Gamma)}+1$. For other values of $s$ they
are mutually exclusive.

\section*{Acknowledgements} This research would not have happened without a prolonged interaction with
Curt McMullen whose advice and ideas were of inestimable value. I have also benefited from discussions
with Larry Rolen, Ian Wagner,  Florin Radulescu, Haakan Hedenmalm, Dan Freed, Terry Gannon and Gaven Martin.

\section{Background in von Neumann algebras.}

We begin the paper by giving an account of von Neumann dimension, and self-contained calculations of some von Neumann dimensions which are slightly different from the
calculations of \cite{ghj},\cite{as},\cite{radulescu2} and \cite{radulescu1}, and require no knowledge of the discrete series for \pslr.

A von Neumann algebra $M$ is a *-closed unital algebra of bounded operators on a (complex) Hilbert space $\h$ which is closed under
the topology of pointwise convergence on $\h$. The  The \emph{commutant} $M'$ of $M$ 
is the algebra of all bounded operators that commute with $M$. It is also a von Neumann algebra and has the same 
centre as $M$.

A vector $\xi\in\h$ is called \emph{cyclic} for $M$ if $M\xi$ is dense in $\h$ and \emph{separating} for $M$ if 
$x\mapsto x\xi$ is injective on $M$. $\xi$ is cyclic for $M$ iff it is separating for $M'$. $\xi$ is said to be a \emph{trace vector}
for $M$ if $\langle ab\xi,\xi\rangle=\langle ba\xi,\xi\rangle$ for all $a,b\in M$.

$M$ is a factor if it  is central, i.e. the centre is $\mathbb C id$. The most obvious factor is the algebra $\bh$ of all bounded operators. A factor $M$ is
called finite if it possesses a trace functional $tr:M\rightarrow \mathbb C$  with the properties
\begin{enumerate}
\item $tr(ab)=tr(ba)$ for all $a,b\in M$.
\item $tr(1)=1.$
\end{enumerate}
The functional is completely determined by these properties. It is positive definite, which means
that $tr(a^*a)>0$ for $a\neq 0$ so one may form the Hilbert space $L^2(M)$ which is the completion
of $M$ with respect to the pre-Hilbert space inner product $\langle a,b\rangle=tr(b^*a)$.
  An easy example is
the $n\times n$-matrices acting on a Hilbert space of dimension $mn$ with some ``multiplicity'' $m$. We will see
more interesting examples very soon. An infinite dimensional finite factor is called a \underline {II$_1$ factor}.

\begin{definition} Let $\Gamma$ be a (countable) discrete group. The \emph{von Neumann algebra of $\Gamma$},
which we will write $vN(\Gamma)$,  is the von Neumann algebra on $\ell^2(\Gamma)$ generated by the left regular
representation $\gamma \mapsto \lambda_\gamma$, where $\lambda_\gamma (f)(\gamma')=f(\gamma^{-1}\gamma')$.

More generally if $\omega:\Gamma\times\Gamma\rightarrow \mathbb T$ is a unit circle valued 2-cocyle, $vN_\omega(\Gamma)$ is generated on $\ell^2(\Gamma)$ by the unitaries $\lambda_\gamma^\omega$ where
$\displaystyle \lambda^\omega_\gamma (f)(\gamma')=\omega(\gamma,\gamma')f(\gamma')$
(so group multiplication is "twisted" by a cocyle).

\end{definition}

It is well known (\cite{takesaki},\cite{dixmier}) that $vN(\Gamma)$ is a \tf iff $\Gamma$ is icc, i.e. all nontrivial conjugacy classes of $\Gamma$
are infinite. $vN_\omega(\Gamma)$ is a \tf if $\Gamma$ is icc. The trace on $vN_\omega(\Gamma)$  is
given by $$tr(\lambda^\omega_\gamma)=\begin{cases} 1 & \mbox{ if } \gamma=id \\
                                                                                          0 & \mbox{ otherwise}
                                                                                          \end{cases}$$

\section{The von Neumann dimension.}

Let $M$ be a finite factor. We will assign a positive real number, or $\infty$, which we will call $dim_M(\mathcal H)$ to any (separable) Hilbert space on
which $M$ acts. It will completely characterise the Hilbert space as a (Hilbert space) $M$-module up to unitary equivalence. In the finite dimensional case it will be $m\over n$ where $m$ is  the multiplicity above, thus measuring
in some sense the number of copies of the $M$-module $M$ inside $\h$.

 A type II$_\infty$ factor is  the closure of the algebra of all finitely supported matrices with entries in a fixed \tf $M$ acting on the direct sum of infinitely many copies of
 the Hilbert spaces on which $M$ acts. A II$_\infty$ factor has a ``trace'' given by adding up the traces of the diagonal
 matrix entries. It is not defined everywhere but one may talk of ``trace class'' operators in a II$_\infty$ factor just
 as one does for $\bh$ (\cite{RS}).
 If $M$ is a \tf on $\h$ we can ``amplify'' it to act diagonally on
  $\oplus_{n=1}^\infty \h$. Its commutant is then a II$_\infty$ factor. All II$_\infty$ factors arise in this way.
 
We will now assume basic facts about type II$_\infty$ factors, traces on them  and comparison of projections in a factor. See \cite{takesaki},\cite{dixmier}.

\begin{proposition}
If $\mathcal H$ is any Hilbert space on which $M$ acts then there is an $M$-linear isometry $$u:\mathcal H\rightarrow \oplus_{n=1}^\infty L^2(M)$$
\end{proposition}
\begin{proof}
$M$ acts diagonally on the direct sum $\mathcal H \oplus (\oplus_{n=1}^\infty L^2(M))$. The commutant $M'$ contains the two projections $p=1\oplus 0$
and $q=0\oplus 1$. Since the commutant is a II$_\infty$ factor and $q$ is certainly infinite, we obtain a partial isometry $u\in M'$ such
that $u^*u=p$ and $uu^*=q$. Identifying $\mathcal H$ with the image of $p$, we have our $u$.
\end{proof}

Note that if  $v$ is any other $M$-linear isometry as above then $vv^*$ is equivalent in $M'$ to $uu^*$. 
Note also that, on $ \oplus_{n=1}^\infty L^2(M)$, the commutant $M'$ admits a canoncially normalised trace $tr_{L^2}$ such that the trace of
any projection onto one of the $L^2(M)$'s is equal to 1. 

\begin{definition} With notation as above $$dim_M(\mathcal H)= tr_{L^2}( uu^*).$$
\end{definition} 

\begin{notes} \rm{
\begin{enumerate}

\item Observe that if $M$ is  the scalars $\mathbb C$ then this definition gives exactly the usual definition of the dimension $dim \mathcal H$ of
a separable Hilbert space. If $M$  is the $n\times n$ matrices we obtain $\displaystyle \frac{dim \mathcal H}{n^2}$.

\item With this philosophy one may canonically normalize the trace on $M'$  by defining $$Tr_{M'}(a)=tr_{L^2}(uau^*)$$
It is not hard to show that $Tr_{M'}$ is $dim_M(\mathcal H)$ times the normalised trace on $M'$. Further if $a:\mathcal H\rightarrow \mathcal K$
and $b:\mathcal K\rightarrow \mathcal H$
is a bounded linear map between Hilbert spaces over $M$ then $$Tr_{M'}(ab)=Tr_{M'}(ba)$$

\item
Our definition is not the same as that of Murray and von Neumann in chapter X of \cite{mvn1} where it  measures the
relative mobility of $M$ and $M'$ as follows.  Take any non-zero $\xi\in \mathcal H$ and consider the 
two closed subspaces $\overline {M\xi}$ and $\overline {M'\xi}$ of $\mathcal H$ with orthogonal projections $p$ and $q$ 
respectively. Clearly $p\in M'$ and $q\in M$ so we may form the ratio $\displaystyle \frac{tr_{M}(q)} {tr_{M'}(p)}$.
This was shown in \cite{mvn1} to be independent of $\xi$. With this fact one may easily show it is equal to our $dim_M(\mathcal H)$. This ratio became known as the ``coupling constant'' but calling it the von Neumann
dimension is more revealing. One reason it is a little obscure in \cite{mvn1} is that the authors defined a theory for
all types (I,II$_1$, II$_\infty$ and III) of factors, each one requiring its own treatment.
\end{enumerate}
}
\end{notes}

\subsection{Elementary properties of $\dim_M \h$}
\begin{theorem} \label{couplingproperties} With notation as above,\\
\5
\noindent (\romannumeral 1) $\dim_M(\h)<\infty $ iff $M'$ is a \tf.\\
(\romannumeral 2) $\dim_M(\h) = \dim_M(\kh)$ iff $M$ on $\h$ and $M$
on $\kh$ are
unitarily equivalent.\\
(\romannumeral 3) %If $\h_i$ are (countably many) $M$-modules,
 $\displaystyle \dim_M ({\oplus_i} \h_i) = \sum_i \dim_M \h_i.$\\
(\romannumeral 4) $\dim_M(L^2(M)q) = tr_M(q)$ for any projection $q\in M$.\\
(\romannumeral 5) If $p$ is a projection in $M$,
$\dim_{pMp}(p\h)=tr_M(p)^{-1}\dim_M(\h)$. \5 For the next two
properties we suppose $M'$ is finite, hence a \tf with\\ trace
$tr_{M'}$ (and $tr_{M'}(1)=1$).
\5 \noindent (\romannumeral 6) If $p$ is a projection in $M'$, $\dim_{Mp} (p\h) = tr_{M'}(p) \dim_M \h$.\\
(\romannumeral 7) $(\dim_M\h)( \dim_{M'}\h) =1$.
\5 \noindent (\romannumeral 8) There is a cyclic vector for $M$ iff $dim_M \mathcal H \leq 1$.
\5 \noindent (\romannumeral 9) There is a separating vector, indeed a trace vector, for $M$ iff $dim_M \mathcal H \geq 1$.
\5 \noindent (\romannumeral 10) If $p\xi=\xi$ for $\xi\in \h$ and $p$ a projection in $M$ then $dim_M(\overline{M\xi})\leq tr_M(p)$.
\end{theorem}

\begin{proof}
These are all standard results due to Murray and von Neumann-\cite{mvn1}. For proofs based on
our definition see \cite{ghj} or \cite{jonesnotes}. $(\romannumeral 10)$ is easiest proved using the Murray von Neumann definition. Clearly
one can reduce to the case $\h=\overline{M\xi}$ and then $M'\xi=M'p\xi=pM'\xi\subseteq p\h$ so if $q$ is 
projection onto $\overline{M'\xi}$,
$tr_M(q)\leq tr_M(p)$.
\end{proof}

\begin{proposition}\label{imprimitivity}
Let $\Gamma$ be an icc discrete group and $\gamma \mapsto v_\gamma$ be a projective unitary group representation on $\mathcal H$ with 2-cocycle $\omega$. 
Suppose there is a projection $q$ on $\mathcal H$ such that 
$$v_\gamma q v_\gamma^{-1}\perp q \quad \forall \gamma\in \Gamma, \gamma \neq id, \qquad \mbox{       and       }\qquad \sum_{\gamma\in \Gamma} v_\gamma q v_\gamma^{-1}=1 $$
then there is a $\Gamma $-linear unitary $U:\mathcal H\rightarrow \ell^2(\Gamma)\otimes q\mathcal H$ with $Uv_\gamma U^{-1}=\lambda_\gamma^\omega\otimes id$
for $\gamma \in \Gamma$.
\end{proposition}
\begin{proof} Choose an orthonormal basis $\{\eta_i|i=1,2,3,\cdots \}$ of $q\mathcal H$. Then by the two conditions of the proposition 
$\{v_\gamma \eta_i |\gamma \in \Gamma, i=1,2,3,\cdots\}$ is an orthonormal basis for $\mathcal H$. Defining $U$ by
 $U(v_\gamma \eta_i)=\varepsilon _\gamma \otimes \eta_i$ gives the desired unitary where $\varepsilon_\gamma$ is the characteristic function of
 $\{\gamma\}$ in $ \ell^2(\Gamma)$.
\end{proof}

\begin{corollary}\label{gammaII1} Suppose $\Gamma,v,q,\omega$ and $U$ are as in proposition \ref{imprimitivity}. Then the action of $\Gamma$ on $\mathcal H$ makes it
into a $vN_\omega(\Gamma)$-module and if  $p$ is a projection on $\mathcal H$ commuting with
$v_\gamma$ for all $\gamma$ then $$dim_{vN_\omega(\Gamma)} \mathcal H= Tr_{B(\mathcal H)}(pqp)=Tr_{B(\mathcal H)}(qpq)$$ where 
$Tr_{B(\mathcal H)}$ is the usual trace  (\cite{RS}, sum of the diagonal elements for a positive operator) on $B(\mathcal H)$.
\end{corollary}
\begin{proof}
The commutant $M'$ of $vN_\omega(\Gamma)$ on $\ell^2(\Gamma)\otimes q\mathcal H$ is the tensor product of $vN_\omega(\Gamma)'$  and $B(q\mathcal H)$
and the correctly normalised trace on it is the tensor product of the trace on $vN_\omega(\Gamma)'$ (on $\ell^2(\Gamma)$) 
and the usual trace on $B(q\mathcal H)$.
Thus since $\varepsilon_{id}$ is a trace vector for for $vN_\omega(\Gamma)'$, for $x\geq 0 \in M'$,
 $$Tr_{M'}(x)=\sum_i\langle x(\varepsilon_{id}\otimes \eta_i),\varepsilon_{id}\otimes \eta_i\rangle$$
 $$=Tr_{B(\ell^2(\Gamma)\otimes q\mathcal H)}(exe)$$
 where $e$ is orthogonal projection onto $\varepsilon_{id}\otimes q\mathcal H$.
 
 Now $Up$ is a $vN_\omega(\Gamma)$-linear isometry from $\mathcal H$ to $\ell^2(\Gamma)\otimes q\mathcal H$ so that, by the definition of von Neumann dimension,
$$dim_{\vgm} p\mathcal H =Tr_{B(\ell^2(\Gamma)\otimes p\mathcal H)}(eUpU^*e)$$
But $U^*eU=q$ so that $$dim_{\vgm} p\mathcal H =Tr_{B(\mathcal H)}(qpq).$$

\end{proof}

A commonly encountered situation in which the hypotheses of \ref{imprimitivity} are satisfied is when $\Gamma$ acts as deck transformations
for a covering space  $\pi:\mathcal M \rightarrow \mathcal N$ between manifolds. Then if $\Gamma$ preserves a smooth measure and $D$
is a fundamental domain, \ref{imprimitivity} applies to the Hilbert space $\mathcal H = L^2(\mathcal M)$ together with the projection $q$ onto $L^2(D)$.
This is the setup for Atiyah's covering space $L^2$ index theorem \cite{a}. We will use it in a slightly modified form where the natural measure is not preserved.
\begin{remark}\rm{
A rather different use of von Neumann dimension occurs in \cite{jo1}. Given a subfactor $N$ of a II$_1$ factor $M$ the
Hilbert space $L^2(M)$ is a left $N$-module and one defines $[M:N]=dim_N(L^2(M))$. Although the von Neumann
dimension itself takes on all positive real values, it turns out that $[M:N]$ must be, if finite, in the set 
$\displaystyle \{4cos^2\pi/n:n=3,4,5,\cdots\}\cup [4,\infty)$. One recognises the squares of the numbers in the
usual generators of the Hecke groups (see \cite{gannontriangle}). 

The context of this paper originated in 1982 in an attempt to find a relation between the Hecke groups and subfactors. That is still 
a long way off as is the attempt to exploit the rich structure of modular forms for a Fuchsian group to produce
``exotic'' subfactors like those of \cite{AH} (see also \cite{JMS} for more examples and details).

}
\end{remark}

\section{Fuchsian groups and $L^2$ holomorphic functions on $\mathbb H$.}\label{fuchsian}

A \textbf{Fuchsian group} $\Gamma$ is by definition a discrete finite covolume subgroup of $PSL_2(\mathbb R)$.
(Finite covolume is not always assumed in the literature.) If $\Sigma$ is a compact Riemann surface of
genus $\geq 2$, its universal covering space is the upper half plane $\mathbb H$ (as a complex manifold). $PSL_2(\mathbb R)$ is the group of
complex automorphisms of $\mathbb H$ so $\pi_1\Sigma$ is a cocompact Fuchsian group. It is also icc. The unit disc $\mathbb D$ is holomorphically the 
same as $\mathbb H$ under the Cayley transform $C:\mathbb H\rightarrow \mathbb D$:
$$ C(z)= \frac{z-i}{z+i}, \qquad C^{-1}(w)=\frac{w+1}{i(w-1)}$$

And the action $g(z)= \displaystyle \frac{az+b}{cz+d}$ for $g=\begin{pmatrix} 
a & b \\
c & d 
\end{pmatrix}$ in \pslr \hspace{2pt} becomes, after conjugation by $C$,
$w\mapsto$.   The action of \pslr \hspace{2pt} on $\mathbb H$ preserves the measure $\mu_0=\displaystyle\frac{dxdy}{y^2}$ which is the 
measure from a hyperbolic metric of constant curvature $-1$. On $\mathbb D$ 
 the measure becomes $\displaystyle \nu_0=4\frac{dxdy}{(1-|w|^2)^2}$.

\begin{proposition}\label{imaginarybehaviour}
For $g$ as above $\displaystyle Im (g(z))=\frac{Im (z)}{|cz+d|^2}$
\end{proposition} 
 All Fuchsian groups are icc (\cite{ake}). If $\Gamma$ is a Fuchsian group it has a fundamental domain which means that $L^2(\mathbb H, d\mu_0)$ satisfies the hypotheses
of \ref{imprimitivity} so that $\Gamma$ generates a II$_1$ factor with II$_\infty$ commutant on $L^2(\mathbb H, d\mu_0)$.

For each real $s>1$ we define the measure $\mu_s=y^{s-2}dxdy$ on $\mathbb H$. $\mu_s$ is not invariant under \pslr
 but   we 
 have, for any $L^1$ function $F$,  
 $$ \int_{\mathbb H}F(z)Im (z)^s\frac{dxdy}{y^2}= \int_{\mathbb H}F(g(z))Im(g(z))^s\frac{dxdy}{y^2}= \int_{\mathbb H}F(g(z))\frac{y^s}{|cz+d|^{2s}}\frac{dxdy}{y^2}$$

 so that, 
 choosing a branch of $(cz+d)^{s}$
 for each $g$, $$(\check\pi_s(g^{-1})f)(z)=\frac{1}{(cz+d)^{s}}f(g(z))$$ defines
a unitary operator on $L^2(\mathbb H, d\mu_s)$, preserving holomorphic functions.

\begin{remark}\label{choiceofbranch}\rm{
For definiteness we will choose the following branch of log to define $(cz+d)^s$:
$$log(cz+d)= \int_\kappa \frac{c}{cz+d} dz+\frac{i\pi}{2} \mbox{ where } \kappa \mbox{ is the straight line from } i \mbox{ to } z$$
Exercise: show that $\check\pi_s(g^{-1})=\check\pi_s(g)^{-1}$.
}
\end{remark}

(The reason for using $s$ rather than $\alpha=s-2$ is that the measure $(1-r^2)^{-2}rdrd\theta$ is  hyperbolic measure,
invariant under the usual action of \pslr on the disc, which is more natural when it comes to Fuchsian groups than the usual 
Lebesgue measure for Bergman space.)    

Now if we consider the function $j:SL_2(\mathbb R)\times\mathbb H\rightarrow \mathbb C$ defined by $\displaystyle j(g,z)=cz+d$, it
is easy to check the cocycle condition $$j(gh,z)=j(g,h(z))j(h,z)$$ so that if $s$ is equal to a \textbf{positive integer} $p$,
the map $g\mapsto \check\pi_p (g)$ defines a unitary representation of $SL_2(\mathbb R)$ which preserves holomorphic functions. 

 If $p$ is even, $\check\pi(-id)$ is the identity so that $\check \pi$  passes to \pslr. If $p$ is odd, $\check\pi(-id)=-id$ so  $\check \pi$ is only
 a projective representation. 
 
\begin{remark}\label{projective}

 \rm{For an arbitrary real postive $s$,  $\gamma\mapsto \check\pi(\gamma)$ is a projective unitary representation.
To see this just take the $s$th. power of the cocycle relation for $j$ above to obtain that 
$j(gh,z)^s$ and $j(g,h(z))^sj(h,z)^s$ differ by a complex number of absolute value equal to one.

 %To see this consider the ``Koopman'' representation on $L^2(\mathbb H, y^{s-2}dxdy)$. This is the unitary 
 %representation $\displaystyle(u_{\gamma^{-1}}f)(z)=\frac{1}{|cz+d|^s}f(\gamma (z))$. We have
 %that $(cz+d)^s=K(z)|cz+d|^s$ for some function $K(z)$. But $K(z)$ is manifestly continuous and varies in
 %a discrete set, so is a constant of absolute value equal to one. NO!
 
 }
 \end{remark} The projective representation $\check \pi$ cannot, for non-integral $s$, be lifted to an honest representation of
 \pslr \2since then it would be a discrete series representation which it isn't-see \cite{barg}.
 But when restricted to $\Gamma$ the relevant cohomology obstruction may vanish (this is the case for \pslz) so 
 one may still get an honest representation of $\Gamma$. That there are Fuchsian groups for which the relevant 
 cohomology obstruction does not vanish will be treated in appendix 1.
 
 If $s$ is not an integer the cocycle condition for $j$ does not imply a cocycle condition for $j^{-p}$ so one only obtains a 
 projective representation for  $\check \pi$.
It can
be considered a unitary representation of the universal cover of \pslr \hspace{2pt}via Bargmann \cite{barg}.

\begin{proposition} \label{cayleyonfunctions}If $f\in L^2(\mathbb H, d\nu_s)$ then $f\mapsto \check f$ where $\displaystyle\check f(z) = \Big(\frac{2}{z+i}\Big)^sf\Big(\frac{z-i}{z+i}\Big)$ defines a unitary from $L^2(\mathbb D, d\nu_s)$ to $L^2(\mathbb H, d\mu_s)$ which intertwines the two projective representations of \pslr.
\end{proposition}
\begin{proof} This can be proved by extending the action on functions from $SL(2,\mathbb R)$ to $SL(2,\mathbb C)$ and conjugating
by the Cayley transform. Unitarity can be checked directly.
\end{proof}

\begin{definition} Let $P_s$ be orthogonal projection from $L^2(\mathbb H, d\mu_s)$ onto the closed subspace spanned by functions
which are holomorphic. This subspace is the ``weighted Bergman space'' $A^2_\alpha$ with $\alpha = s-2.$ We will use the notation indifferently
for functions on $\mathbb D$ or $\mathbb H$.  The projective representation $\pi_s$ of \pslr is defined to be
the restriction of $\check \pi$ to $A^2_{s-2}$. 
\end{definition}

\begin{remark}\label{reproducing}\rm{
These Hilbert spaces of analytic functions are "reproducing kernel" Hilbert spaces.  
The parameter in the literature is usually $\alpha=s-2$. This means that for each $z\in \mathbb H$
there is a $\varepsilon_z\in A^2_{s-2}$ such that $$ \langle \varepsilon_z,f \rangle=f(z)$$
This follows from the continuity of point evaluation.}
\end{remark}
 As noted, if $s$ is an even positive integer we get an honest unitary representation of \pslr, but not for $s$ odd. 

%\begin{proposition} If $p$ is a positive even integer, the representation $\pi_p$ of $\Gamma$ extends to
%a representation of $vN(\Gamma)$ and $$dim_{vN(\Gamma)}\mathfrak H_p=Tr_{B(\mathfrak H_p)}(pqp)$$ where $q$ is orthogonal
%projection onto $L^2(F)$, $F$ being a fundamental domain for $\Gamma$ on $\mathbb H$ (or $\mathbb D$ for $\mathcal {H}_p$).
%\end{proposition}
%\begin{proof} The hypotheses of \ref{gammaII1} are satisfied.
%\end{proof}
Let $\Gamma$ be a Fuchsian group with fundamental domain $F$. We have seen that $\check \pi_ s$ restricted to $\Gamma$ defines a projective
unitary representation of the II$_1$ factor $M=vN_\omega(\Gamma)$ where $\omega$ is the $2$-cocycle with values in the circle 
which comes from the chosen branch of the logarithm of $cz+d$ on $\mathbb H$.
To calculate the von Neumann dimension $dim_M(A^2_{s-2})$ for we  will use an orthonormal basis of $A^2_{s-2}$. We will work in
$\mathbb D$ where it is obvious that the powers of $z$ are orthogonal so all we need to do is normalize them.
The result is very well known (see \cite{ghj},\cite{radulescu1}) but we include the calculation for the convenience of the reader.

\begin{proposition}\label{onb} \mbox{  } \\
\begin{enumerate}\item Let $\displaystyle e_n(w)=\sqrt{{\frac{s-1}{4\pi}}}\sqrt{\frac{s(s+1)....(s+n-1)}{n!}}w^n$ for $w\in \mathbb D$. 
Then $e_n$ is an orthonormal basis for $A^2_{s-2}$.

\item Let $\displaystyle f_n(z)=\sqrt{{\frac{s-1}{4\pi}}}\sqrt{\frac{s(s+1)....(s+n-1)}{n!}}\Big(\frac{2}{z+i}\Big)^s\Big(\frac{z-i}{z+i}\Big)^n$ for $z\in \mathbb H$. 
Then $f _n$ is an orthonormal basis for $A^2_{s-2}$.
\end{enumerate}
\end{proposition}
\begin{proof}
It is trivial that $\langle e_n,e_m\rangle=0$ for $n\neq m$, so we only need to calculate, writing $w=u+iv$, 
$$||w^n||^2=\int_{\mathbb D}|w|^{2n} (1-|w|^2)^{s-2}4dudv=4\int_0^{2\pi}\int_0^1r^{2n}(1-r^2)^{s-2}rdrd\theta$$
Putting $t=r^2$ we get $$4\pi\int_0^1t^n(1-t)^{s-2}dt=4\pi\beta(n+1,s-1)=4\pi\frac{\Gamma(n+1)\Gamma(s-1)}{\Gamma(n+s)}$$
Expanding the $\Gamma$ functions we get the result for $e_n$ and the result for $f_n$ follows from \ref{cayleyonfunctions}
\end{proof}

\begin{theorem}\label{vonNdimension} With notation as above $$dim_M(A^2_{s-2})=\frac{s-1}{4\pi} covolume(\Gamma)$$
\end{theorem}
\begin{proof}
We will do the calculation in the $\mathbb D$ model.
By \ref{gammaII1} we have to calculate 
$$\sum_{n=0}^\infty \int_F |e_n(w)|^2(1-|w|^2)^{s-2}4dudv=\frac{s-1}{4\pi}\sum_{n=0}^\infty \int_F  \frac{s(s+1)....(s+n-1)}{n!}r^{2n}4dudv $$
Everything in sight is positive so  one can commute summing and integration. 
We have $\displaystyle  (1-r^2)^{-s}=\sum_{n=0}^\infty \frac{s(s+1)...(s+n-1)}{n!}r^{2n} $ which gives
$\displaystyle dim_M(A^2_{s-2})=\frac{s-1}{4\pi}\int_F\frac{dudv}{v^2}$ so we get $$dim_M(A^2_{s-2})=\frac{s-1}{4\pi}\mbox{covolume} (\Gamma)$$
as required.

\end{proof}
%For other ways to obtain this result see \cite{ghj},\cite{radulescu1}.

\begin{notes} {Special cases.}
\begin{enumerate}
\item $\Gamma=$\pslz . Here the covolume(=hyperbolic area of fundamental domain) is, by Gauss-Bonnet or direct integration over $F$, equal
to $\pi/3$. So we have, for $s>1$, $$dim_{vN(PSL_2\mathbb Z)}A^2_{s-2}=\frac{s-1}{12}.$$
Since $\Gamma$ is in this case the free product of two cyclic groups the projective representation actually
lifts to an honest one so we are dealing with $vN(PSL(2,\mathbb Z))$.
\item If $\Sigma$ is a compact Riemann surface of genus $g>1$ with hyperbolic metric, its area is $4\pi(g-1)$ so 
 $$dim_{vN(\pi_1(\Sigma))}A^2_{s-2}=(s-1)(g-1).$$
 In this case the projective representation does not necessarily lift to an honest one as we will show in 
 appendix \ref{obstruction}. However if $s$ is an odd integer the existence of spin structures shows that 
 the lifting does exist.
\end{enumerate}
 \end{notes}
 
 Why might these von Neumann dimension formulae actually lead to non-trivial results?
 The fact that equality of traces in a factor implies equivalence of projections is an ergodic theoretic result ultimately
 relying on patching together lots of little projections. There are some instances of results using it which are nontrivial.
 Let us discuss the author's favourite (due to Kaplansky). In fact it does not even use factoriality!
 
 \begin{theorem} Let $\Gamma$ be a discrete group and $\mathbb F$ a field of characteristic zero. Let $\mathbb F\Gamma$ be the 
 group algebra. Then  $ab=1 \iff ba=1$ in $\mathbb F\Gamma$.
 
 \end{theorem} 
 \begin{proof}Since the relations $ab=1$ and  $ba=1$ only involve finitely many scalars we may embed $\mathbb F$ in $\mathbb C$ and work in
 $\mathbb C \Gamma$ which embeds into $vN(\Gamma)$.

 So the result follows from $ab=1 \iff ba=1$ in a finite von Neumann algebra $M$ with trace $tr$. Let $M$ act on some $\mathfrak H$.
 
 Suppose $ab=1$. Then for any $\xi\in \mathfrak H$, $ba(b\xi)=b\xi$ so since $ba$ is bounded it suffices to show that the 
 range of $b$ is dense. But if $b=u|b|$ is the polar decomposition of $b$ then $u$ is a partial isometrey from the orthogonal
 complement of the kernel of $b$ to the closure of the image of $b$. But $u^*u=1$ since $ker(b)=0$ (since $ab=1$).
 So $tr(uu^*)=1$ is one which means $uu^*=1$ so the image of $b$ is dense.
 \end{proof}
 
 The conclusion of the theorem remains an open problem if one drops the condition that the characteristic of
 the field be zero. Thus the use of von Neumann algebra in this context can have considerable content and it
 could be that the results of this paper are quite difficult to obtain by any other means.
 Here is a sample (it will be one direction of theorem \ref{ultimate} below).
 
 \begin{proposition} \label{sample}Let $\Gamma$ be a Fuchsian group. Then if $\displaystyle s>1+\frac{4\pi}{covolume(\Gamma)}$ and $z\in\mathbb H$
there is a non-zero function in $A^2_{s-2}$  vanishing the orbit $\Gamma(z)$.
 
 \end{proposition}
 \begin{proof}
 ($\Gamma$ is icc by \cite{ake}. Let $\epsilon_z$ be the reproducing kernel vector for $z$ so
 that $\langle \epsilon_z,\xi\rangle=\xi(z)$ for all $\xi\in A^2_{s-2}$. Then the von Neumann dimension of the closure of $\vgm \epsilon_z$ is at most
 $1$ by ($\romannumeral 8$) of \ref{couplingproperties}. But by  \ref{vonNdimension}, the von Neumann dimension of $A^2_{s-2}$ is greater than $1$. So there
 is a $\xi \in A^2_{s-2}$ orthogonal to $\pi_s(\gamma) \epsilon_z$ for every $\gamma$. Thus $\xi$ vanishes on $\Gamma(z)$.

 \end{proof}
 
 By a relatively subtle argument with cusp forms it is possible to prove this result without the use of von Neumann
 algebras for \pslz (see the discussion after definition \ref{cuspbounded}), but a proof in full generality might be
 very complicated.
 
\section{Wandering vectors and trace vectors.}
 
 For convenience we introduce the following definition which appears to be well accepted.
 
 \begin{definition}\label{wandering}  If $\pi$ is a (projective unitary) representation of a group $\Gamma$ on a Hilbert space $\mathcal H$ then\\
 \begin{enumerate}
 \item A (non-zero) vector $\xi\in \mathcal H$ is called a \textbf{wandering vector} for $\pi$ if 
 $$\langle \xi, \pi(\gamma) (\xi)\rangle =0 \mbox{ for all } \gamma \neq 1 \mbox{ in } \Gamma$$
 \item A subspace $V\subseteq \mathcal H$ is called a \textbf{wandering subspace} if $$\pi(\gamma) (V)\perp V \mbox{ for all }  \gamma \neq 1 \mbox{ in } \Gamma$$
 \end{enumerate}
 \end{definition}

Note that any nonzero element of a wandering subspace is a wandering vector and orthogonal vectors in a wandering subspace produce
wandering vectors with orthogonal orbits.

\begin{definition}If $M$ is a  von Neumann algebra on $\h$, a non-zero vector $\eta\in\h$ is called a \emph{trace vector} for $M$ if
$\langle x\eta,\eta\rangle$ is a multiple of the trace of $x$ for every $x\in M$.
\end{definition}
There is a simple relationship between the two concepts:
\begin{proposition}\label{standardsubspace} If $\pi$ is a (projective unitary) representation of a group $\Gamma$ on a Hilbert space $\mathcal H$ and $\xi$ is a wandering vector for $\pi$ then $\xi$ is a trace vector for the von Neumann algebra $M$ generated by
$\pi(\Gamma)$. Moreover on the closure of the subspace $[\pi(\Gamma)\xi]$, $M$ is isomorphic to the twisted 
group von Neumann algebra $vN_\omega(\Gamma)$ (where $\omega$ is the 2-cocycle of the projective representation),
acting on $L^2(M)$.
\end{proposition}
\section{Proof of the main theorem.}
We will use the following easy result on Bergman space functions (``popping zeros''):
\begin{lemma}\label{popping} Let $f\in A^2_{s-2}$ be a nonzero function with a zero of order $k$ 
at $w$, i.e. $f^{(j)}(w)=0$ for $j=0,1,2,\cdots, k-1$ but $f^{(k)}(w)\neq 0$. Then the functions 
$(z-w)^{-j}f(z)$ for $j=1,2,\cdots,k$ are in $A^2_{s-2}$.
% are linearly independent and have the same other zeros of the
%same order as $f$.
%function $g\in A^2_{s-2}$ with $g(z)\neq 0$ but vanishing at all the other zeros of $f$.

\end{lemma}
\begin{proof} By the transitivity of the action of \pslr ($=SU(1,1)$) we may assume $z=0$.
%Let $k$ be the order of the zero of $f$ at $0$.
Write $$f = \sum_{n=k}^\infty c_n e_n$$ where
 $\displaystyle e_n(z)=\sqrt{{\frac{s-1}{4\pi}}}\sqrt{\frac{s(s+1)....(s+n-1)}{n!}}z^n$ are the orthonormal basis constructed in \ref{onb}. 
We know that $c_n$ is square summable. The limit of the sequence
$$a_n=\sqrt{\frac{s(s+1)....(s+n+k-1)}{s(s+1)....(s+n-1)} \frac{n!}{(n+k)!}}$$
is $1$ so $a_n$ is bounded.
The holomorphic function $g(z)=z^{-k}f(z)$ has Taylor series $$\sum_{n=0}^\infty c_{n+k}z^{-k}e_{n+k}(z)=\sum_{n=0}^\infty a_nc_{n+k}e_{n}(z)$$
Thus  $g(z)\in A^2_{s-2}$ and so is $z^{j}g(z)$ for $1\leq j<k$.
% But multiplying by 
%a power of $z$ doesn not change the nature of any of the other zeros of $f$. Linear independence follows from
%that of the powers of $z$.

\end{proof}
 Now let $\Gamma$ be an \emph{orderable}  Fuchsian group and let $\gamma \mapsto \pi_s(\gamma)$ be the (projective) unitary 
representation on $A^2_{s-2}$ that we have been considering.

( Recall that a group  $\Gamma$ is \emph{orderable} if it admits a total order $<$ which is invariant under
 left translation, i.e. $\alpha < \beta \iff \gamma\alpha<\gamma\beta$ for all $\gamma$. Free groups are orderable
 as are fundamental groups of surfaces-\cite{rolfsenwiest}.)

\begin{theorem}\label{wandering}
Suppose $O_1,O_2,\cdots, O_n$ are disjoint orbits in $\mathbb D$ of  $\Gamma$. 
Let $f\in A^2_{s-2}$ be non-zero, with a zero of  order at least $v_i$ on all points of $O_i$ 
%and of order exactly $v_i$ at some point in $O_i$. 
Then there is a wandering subspace W 
of dimension  $t=\sum_{i=1}^nv_i$ for $\pi_s(\Gamma)$, and $\pi_s(\gamma)(f)\in W^\perp\quad \forall \gamma$.
\end{theorem}
\begin{proof}
To make the argument clear let us begin with the case of the orbit of a single point $z$, with $f$ having zeros at 
$\gamma(z) \quad \forall \gamma\in\Gamma$.  

Choose a left ordering $<$ of $\Gamma$ and  
define the closed subspaces $U$ and $V$ of $A^2_{s-2}$ to be  $$U=\{\xi |\xi(\gamma(z))=0 \mbox{ for }\gamma \leq id\}$$ and
$$V=\{\xi |\xi(\gamma(z))=0 \mbox{ for }\gamma <id\}$$

We will now show that  a vector in the orthogonal complement $U^\perp \cap V$ of $U$ in $V$ is a wandering vector for $\Gamma$.  

For suppose $\xi \in U^\perp \cap V$. Then for $\gamma<id$ and any other $\lambda \leq id$, 
$$\gamma \lambda \leq\gamma id =\gamma<id$$ so 
$$\pi_{s}(\gamma^{-1})\xi(\lambda(z))=\frac{1}{(cz+d)^s}\xi(\gamma\lambda(z))=0 \mbox{     (since }\xi \in V )$$ 
which means that $\pi_{s}(\gamma^{-1})\xi\in U$ and thus $$\langle \pi_{s}(\gamma^{-1})\xi,\xi\rangle=0$$ since $\xi\in U^\perp$. This also means, by unitarity,
 $$\langle \pi_{s}(\gamma)\xi,\xi\rangle=0$$
So $\xi$ is wandering.

Moreover for all $\gamma$, $\pi(\gamma)(f)$ vanishes on the entire orbit so it is in $U$, hence it is orthogonal to $\xi$.

So we just need to check that the containment of $U$ in $V$ is strict. For this, divide $f$ enough times by a linear
function vanishing at $z$ as in lemma \ref{popping}.

For the general case we essentially repeat the argument. Fix a "base point" $z_i$ in each $O_i$. Given $f$ satisfying the hypotheses of the theorem
let $$U=\{\xi |\xi^{(j)}(\gamma(z_i))=0 \mbox{ for }\gamma \leq id \quad\forall i\mbox{ and all } 0\leq j< v_i\}$$ and
$$V=\{\xi |\xi^{(j)}(\gamma(z_i))=0 \mbox{ for }\gamma < id \quad\forall i\mbox{ and all } 0\leq j<v_i.\}$$ 
Clearly $U\subseteq V$ and put $W=U^\perp\cap V$. We claim $W$ is a wandering subspace for $\pi_s(\Gamma)$.

 For suppose $\xi,\eta \in U^\perp \cap V$. Then for $\gamma<id$ and any other $\lambda \leq id$, 
$$\gamma \lambda \leq\gamma id =\gamma<id$$ so $$\pi_{s}(\gamma^{-1})\xi(\lambda(z_i))=\frac{1}{(cz+d)^s}\xi(\gamma\lambda(z_i))=0 \quad \forall i$$ 
The factor $\displaystyle \frac{1}{(cz+d)^s}$ doesn't change the nature of the zeros so $\pi_{s}(\gamma^{-1})\xi\in U$ and thus $$\langle \pi_{s}(\gamma^{-1})\xi,\eta\rangle=0$$ since $\eta\in U^\perp$. Which, since $\xi$ and $\eta$ are arbitrary in $W$, also means by unitarity that
 $$\langle \pi_{s}(\gamma)\xi,\eta\rangle=0 \mbox{ for \emph{all} }\gamma.$$
 
Thus $W$ is wandering.
Moreover $\pi_s(\gamma)(f)\in U$ so $\pi_s(\gamma)(f)\perp W \quad \forall \gamma \in \Gamma$. 

We will now show that the dimension of $W$ is at least  $t=\sum_{i=1}^nv_i$.
Wolog we may assume that, for each $i$,  the  order of the zero at $z_i$ of $f$ is exaclty $v_i$. (It suffices to apply
lemma \ref{popping} at each $z_i$ to the given non-zero $f$.) For such an $f$ we claim that the $t$ functions
$$\frac{f}{(z-z_i)^j}\mbox{ for  } i=1,\cdots, n \mbox{ and } j=1,\cdots v_i$$ are in $V$ and are linearly independent  modulo $U$.
They are in $V$ by \ref{popping} and the fact that the nature of the zeros of $f$ on the rest of the orbits is unchanged by
multiplication by powers of $z-z_i$. Suppose $a_{i,j}$ are constants with 
$$ \sum_{i,j} a_{i,j} \frac{f}{(z-z_i)^j}=g \mbox{ for some }g\in U$$
Then $g$ has a zero of order at least $v_i$ at $z_i$ so the meromorphic function $\displaystyle {g\over f}$ is holomorphic at each
$z_i$. This forces all the $a_{i,j}$ to be zero.

% (a subspace of a wandering subspace is wandering). 
\iffalse For this define the linear functionals on $V$ by $$D^j_i(f)=f^{(j)}(z_i)\mbox{ for } i=1,2,\cdots,n \mbox{  and  }j=1,2,\cdots v_i$$
Then  $\underset{i,j} \cap ker(D^j_i) =U$  by definition, so $dim W$ is at most $t$. Together the maps $D^j_i$  form a linear
map from $V$ to $\mathbb C^t$ so it suffices to exhibit $t$ elements of $V$ whose images are linearly 
independent. Note that by lemma \ref{popping} we may multiply $f$ by powers of the $(z-z_i)$ so that
$D^{v_i+1}_i f(z_i)\neq 0 \quad \forall i$. Lemma \ref{popping} shows that the functions $(z-z_i)^{-j}f(z)$ are all in $W$
 for all $i$ and $j<v_i$.
 \fi
 \end{proof}
 
 We now deduce some consequences of theorem \ref{wandering}. We start with \ref{ultimate}
 which is the most straightforward. 
 (The condition on the freeness of the action on the orbit in theorem \ref{ultimate} is
significant. See remark \ref{freeorbit}.) For convenience of reading we recall \ref{ultimate}:

\begin{theorem} \label{ultimate2}Let $\Gamma$ be a Fuchsian group and $\Gamma(z)$ be an orbit in $\mathbb H$
containing no fixed points for any element of $\Gamma$. Then
there is a non-zero function in $A^2_{s-2}$  vanishing on $\Gamma(z)$ iff $$s>1+\frac{4\pi}{covolume(\Gamma)}$$
%(\romannumeral 1)If $\displaystyle s\leq1+\frac{4\pi}{covolume(\Gamma)}$ there is no non-zero function in $A^2_{s-2}$ vanishing on $\Gamma(z)$.
%(\romannumeral 2)If $\displaystyle s>1+\frac{4\pi}{covolume(\Gamma)}$ there is a non-zero function in $A^2_{s-2}$ vanishing on $\Gamma(z)$.
\end{theorem}

\begin{proof} ($\implies$) By \cite{HKS} $\Gamma$ is either a free product 
of finitely many cyclic groups or has a surface group of finite index. Either way there is an orderable subgroup
$\Psi<\Gamma$ with $n=[\Gamma:\Psi]<\infty$. Let $<$ be a left ordering
on $\Psi$.
Suppose by way of contradiction that $f\in A^2_{s-2}$ is non-zero but vanishes
on $\Gamma(z)$. 

%We will exhibit $n$ non-zero elements $\xi_1,\xi_2,\cdots \xi_n\in A^2_{s-2}$ with 
%$$\langle \gamma\xi_i,\xi_j\rangle=0 \mbox{ whenever }i\neq j \mbox{ or }\gamma\neq id.$$
%So that if $M$ is the \tf $vN_\omega(\Omega)$ then each $\xi_i$ is a trace vector and the $M-$ modules 
%$\overline{M\xi_i}$ are mutually orthogonal, of von Neumann dimension one. Hence 
%$n<\dim_M A^2_{s-2}=n \dim_{ vN_\omega(\Gamma)} A^2_{s-2}$  which forces 
%$\dim_{ vN_\omega(\Gamma)} A^2_{s-2}\geq 1$, contradicting the hypothesis (\romannumeral 1).

By the fixed point hypothesis,  $\Gamma(z)$  consists of $n$ disjoint $\Psi$-orbits so
apply theorem \ref{wandering} to obtain a wandering subspace of dimension at least $n$. Choosing an orthonormal basis we obtain  $n$ 
 vectors $\xi_i$ 
 so that if $M$ is the \tf $vN_\omega(\Psi)$ then each $\xi_i$ is a trace vector and the $M-$modules 
$\overline{M\xi_i}$ are mutually orthogonal, of von Neumann dimension one. Hence 
$n\leq\dim_M A^2_{s-2}=n \dim_{ vN_\omega(\Gamma)} A^2_{s-2}$  which forces 
$\dim_{ vN_\omega(\Gamma)} A^2_{s-2}\geq 1$,  in other words $\displaystyle s\geq1+\frac{4\pi}{covolume(\Gamma)}$. 
To see that $s$ cannot be equal to $\displaystyle 1+\frac{4\pi}{covolume(\Gamma)}$, observe that by \ref{wandering},
 $f$ itself is orthogonal to 
$\vgm(V^\perp\cap W)$ which already has von Neumann dimension equal to one, a contradiction.

% \mbox{ whenever } \gamma<id $, for all $j$, and
%  for each  $i=1,\dots n+1$ let 
%$$V_i=\{f\in V| f(z_j)=0 \mbox{ for } j<i \}
%$$
%Clearly $V_{i+1}\subseteq V_i$ and $f$ shows that $V_n$ is non-zero. Successively popping the zeros 
%of $f$ shows that the inclusions $V_{i+1}\subset V_i$ are strict with some $f_i\in V_i\setminus V_{i+1}$. So we set

%$$\xi_i=P_{V_{i+1}^\perp \cap V_i}(f_i)$$ 

 ($\impliedby$) See \ref{sample}.

\end{proof}

\begin{corollary} If $\Gamma$ is an arbitrary Fuchsian group then for all $z$,  $\Omega(\Gamma(z))$ is strictly greater
than $1$.

\end{corollary}
\begin{proof} If $\Gamma_0\subseteq \Gamma$ then $X_{\Gamma(z)}\subseteq X_{\Gamma_0(z)}$
so $\Omega(\Gamma_0(z))\leq \Omega(\Gamma(z))$. As before, \cite{HKS}, any Fuchsian group $\Gamma$ has a
left orderable subgroup $\Gamma_0$ of finite index and the von Neumann dimension multiplies by $[\Gamma:\Gamma_0]$ on restricting to
the $\Gamma_0$. So by part \ref{ultimate2}, $s\in \Omega(\Gamma_0(z))$ for $s$ sufficiently close to $1$.

\end{proof}

The next corollary can be proved by other means, e.g. equidistribution-see [] (McMullen)
\begin{corollary} For any Fuchsian group $\Gamma$ and any $z\in \mathbb D$,
$$\sum_{\gamma \in \Gamma} (1-|\gamma(z)|) \mbox {  diverges. }$$
\end{corollary}

\begin{proof} If the sum converged there would be a Hardy space function vanishing on $\Gamma(z)$ and Hardy
space is contained in the Bergman spaces.

\end{proof} 
In \cite{hkz} a density $D^+(S)$ called the ``upper asymptotic $\kappa$-density'' is defined for subsets $S$ of
the unit disc. It is shown on page 131 of that book that the condition $D^+ (S) \leq\frac{1+\alpha}{p}$ is necessary and the condition $D^+(S) < \frac{1+\alpha}{p}$ is sufficient for $A$ to be an $A^p_\alpha$-zero set.
\begin{corollary}\label{density} If $\Gamma$ and $z$ are as in \ref{ultimate} then
 $$D^+(\Gamma(z))=\frac{2\pi}{\mbox{covolume }(\Gamma)}$$
\end{corollary}
\begin{proof}
Putting $p=2$ in the condition from \cite{hkz} above we get
 $D^+ (\Gamma(z))=\frac{1+\alpha}{2}=\frac{s-1}{2}=\frac{2\pi}{\mbox{covolume }(\Gamma)}$ from \ref{ultimate}.
\end{proof}
\begin{remark} Once we have $D^+ (\Gamma(z))$ we know when $\Gamma(z)$ is an  $A^p_\alpha$-zero set 
for all $p$ by \cite{hkz}. Thus the $L^2$ methods of this paper solve, thanks to \cite{hkz}, an $L^p$ problem for all $p$.
\end{remark}

\section{Use of cusp forms.}

Let us restrict initially to the case $\Gamma=$\pslz. A cusp form of weight $p$ is a function $f:\mathbb H\rightarrow \mathbb C$ which is holomorphic and satisfies $$f(\gamma(z))=(cz+d)^pf(z)$$ which means
 that $f(z+1)=f(z)$ so that we may write $f$ as a function of $q=e^{2\pi i z}$. The cusp form condition is then that 
 $$f(z)=\sum_{n=1}^\infty a_nq^n$$
 
 The first thing to observe is that   $$|f(z)|\leq (Constant)(Im z)^{-p/2}.$$  To see this note that
 $|f(z)|Im(z)^{p/2}$ is invariant under the action
 of \pslz. (Follows from modularity of $f$ and proposition \ref{imaginarybehaviour}.) But since $f(z)=qg(z)$ with $g$ having a finite
 limit as $q\rightarrow 0$, $|f(z)|Im(z)^{p/2}$ is bounded on a fundamental domain, hence everywhere. 
 
 The first cusp form is the modular discriminant $\Delta(z)$ of weight 12 which a function of $q$ can be
 written $\displaystyle q\prod_{n=1}^\infty(1-q^n)^{24}$. It is the 24th power of the Dedekind $\eta$ function.
 Cusp forms give a graded algebra under multiplication and
 can be multiplied by modular forms (same invariance as cusp forms but don't vanish at $\infty$) to give other
 cusp forms. See \cite{serre}.

 Now let $\Gamma$ be an arbitrary Fuchsian group and, following some authors (\cite{zagier}),  we say a cusp form of weight $p$ is a holomorphic function $f:\mathbb H\rightarrow \mathbb C$ such
 that $f(\gamma(z))=(cz+d)^pf(z)$ and $|f(z)|\leq (Constant)(Im z)^{-p/2}.$

 \begin{proposition} \label{bounded} If $f$ is a cusp form of (integer) weight $p$ let $M_f: L^2(\mathbb H, y^{s-2}dxdy)\rightarrow L^2(\mathbb H, y^{s+p-2}dxdy)$
 be the operator of multiplication by $f$. Then $M_f$ is a bounded linear operator intertwining the actions of $\check \pi _{s+p}(\gamma)$
 and $\check \pi _{s}(\gamma)$ , and preserving the
 subspace of holomorphic functions. Also $M_f^*(\xi)(z)=Im(z)^{p}\overline{f(z)}\xi(z)$
 \end{proposition}
 \begin{proof}
 Boundedness: For $\xi\in L^2(\mathbb H, y^{s-2}dxdy)$, using the above bound on $|f(z)|$, 
  $$||M_f\xi||^2=\int_{\mathbb H}|f(z)|^2|\xi(z)|^2y^{s+p-2}dxdy\leq (Constant)\int_{\mathbb H} |\xi(z)|^2y^{s-2}dxdy.$$
 Also
 $$\check \pi _{s+p}(\gamma ^{-1})(M_f \xi)(z)=\frac{1}{(cz+d)^{s+p}}f(\gamma (z))\xi(\gamma (z))=\frac{(cz+d)^p}{(cz+d)^{s+p}}f(z)\xi(\gamma (z))$$
 $$=M_f (\check \pi_s(\gamma^{-1}) \xi)(z).$$
  
 And finally
 $$\langle M_f\xi,\eta\rangle=\int_{\mathbb H} f(z)\xi(z)\overline{\eta(z)}y^{s+p-2}dxdy=\int_{\mathbb H} \xi(z)\overline{\overline{f(z)}\eta(z)}Im(z)^py^{s-2}dxdy$$ which is the formula given in the statement of the 
 proposition for $M_f^*$.
 \end{proof}

\begin{definition}\label{cuspbounded} If $f$ is a cusp form of weight $p$ we call $T_f$ the operator from $A^2_\alpha$ to $A^2_{\alpha+p}$
given by $$T_f=PM_f$$ where $P$ is orthogonal projection from $\xi\in L^2(\mathbb H, y^{\alpha+p}dxdy)$ onto
Bergman space.
\end{definition}
We saw above that for a Fuchsian group $\Gamma$  there is, for $s$ large enough and any $z$, simply because of von Neumann dimension, a function
$f\in A^2_{s-2}$ vanishing on $\Gamma(z)$. But the von Neumann dimension is a blunt tool and is of no help whatsoever
in finding such functions.
Cusp forms give us \textbf{explicit} functions in Bergman spaces vanishing on orbits under $\Gamma$.
 Indeed if $f$ is a cusp form of
weight $p$ vanishing at $z\in \mathbb H$, and $\xi\in A^2_{s-2}$ then by \ref{bounded}, $T_f \xi$ is in $\mathcal {H}_{s+p}$ and vanishes
on $\Gamma(z)$.  %for all $\gamma\in \Gamma$. 
 For $\Gamma =$\pslz this shows that there are elements of $A^2_{s-2}$ vanishing at $\Gamma(e^{i\pi/3})$ provided $s>17$. This is because
the Eisenstein series $G_2$ is a modular form of weight 4 vanishing at  $e^{i\pi/3}$ so that $\Delta G_2$ is a cusp form of weight $16$ vanishing at
$e^{i\pi/3}$. Elements of $A^2_{s-2}$ for $s>1$ may be multiplied by $G_2\Delta$ to give the required Bergman space functions.

L. Rolen and I. Wagner have improved this method considerably (\cite{RW}) to get explicit elements of $A^2_{s-2}$ vanishing on  \pslz$(z)$ for any $s>13$:
Begin with the modular function $j(z)$ and choose any $w\in \mathbb H$. Then $j(z)-w$ is a holomorphic function that vanishes exactly on the \pslz  \hspace{2pt} orbit of a $z_0$ with 
$j(z_0)=w$. Now multiply by $\Delta(z)$ to obtain a modular form vanishing on the same set. Then choose a branch of $\eta(z)^r$ for $r$ real, small and positive.
Then the product $f(z)=(j(z)-w)\Delta \eta(z)^r$ satisifies $|f(\gamma (z)| Im(z)^{6+r/4}=|f(z)| $. Since $(j(z)-w)\Delta(z)$ has a limit as $q\rightarrow 0$ and 
$|\eta(z)|$ tends to zero as $Im z$ grows,  $|f(\gamma (z)| Im(z)^{6+r/4}$ is bounded on a fundamental domain and hence 
$$|f(z)|\leq (constant) Im(z)^{-(6+r/4)}$$
Thus as before, multiplication by $f$ defines a bounded operator from $A_{s-2}$ to $A_{s+10+r/2}$. Choosing $s$ close to $1$ and $r$ close to zero
we get  an \emph{explicit} element of $A^2_{t-2}$ whose zero set is \emph{exactly} the orbit \pslz$(z_0)$ for any $t>13$.

%This shows that $\Omega($\pslz$(z))=11$.

\begin{remark}\label{freeorbit} \rm{Here is an example showing that the freeness of the action on the orbit of $z$ is
essential. Let $G_2$ be the Eisenstein series modular form for $\Gamma=$\pslz of (smallest) weight 4. Then $G_2(e^{\frac{\pi i}{3}} )=0$ so $G_2$
vanishes on the $\Gamma$ orbit of $e^{\frac{\pi i}{3}}$. Using the same trick as above, multiply $G_2$ by
some branch of $\eta(z)^r$ for $r$ real, small and positive. The resulting holomorphic function $f$ will satisfy
$$|f(z)|\leq (constant) Im(z)^{-(2+r/4)}$$ and so defines by multiplication a bounded map from 
$A_\alpha$ to $A_{\alpha+4+r/2}$. So if $s$ is slightly bigger than $1$ we obtain elements of  $A_{3+\epsilon}$
vanishing exactly on the $\Gamma$ orbit of $e^{\frac{\pi i}{3}}$ for all $\epsilon >0$.

}
\end{remark}
\begin{remark}\rm{
For a cusp forms $f$ of weight $p$ the operator $T_f$ is $M$-linear where $M=\vgm$ so if we let $M$ acting
diagonally on the direct sum 
$\displaystyle \oplus_{n=0}^\infty A_{s-2+np}$
  of Bergman spaces, the $T_f$ define operators in the commutant which is a II$_\infty$ factor.
  But we can also think of $T_f$ as a map between Bergman spaces intertwining the action of $M$.
  \begin{proposition} The closure of $T_fA^2_{s-2}$ is an $M$-module of von Neumann dimension
  equal to that of $A^2_{s-2}$.
  
  \end{proposition}
  \begin{proof} This is trivial since multiplication by a non-zero holomorphic function is injective so the polar
  decomposition of $T_f$ gives a unitary equivalence. 
  \end{proof}
  Here is a simple consequence of von Neumann dimension in the spirit of proposition \ref{sample}.
  \begin{corollary} For any $s>1$ and every cusp form $f$ of weight $p$ there is a $\xi \in A^2_{s+p-2}$ which
  is orthogonal to $f\eta$ for all $\eta\in A^2_{s-2}$.
  \end{corollary}
  
  C. Mc Mullen pointed out that this result is trivial if $f$ has zeros since then the Bergman reproducing kernel 
  vector $\epsilon_z$ is automatically orthogonal to $f\eta$ for all $\eta\in A^2_{s-2}$. For a cusp form with no
  zeros, like $\Delta$ we 
  have not seen a constructive proof.
  }
\end{remark}

%Although \pslz is not left orderable so \ref{nobergmanzeroset} does not apply, we speculate that there are no non-zero elements of $A^2_{s-2}$ vanishing on an 
%orbit of \pslz for any $s<13$, but we have no guess for $s=13$. 
\section{Fixed points}
It is possible to improve on theorem \ref{ultimate} by a closer analysis of the orbits of an orderable subgroup
of finite index. We are guided by example \ref{freeorbit}. 
If $\Gamma$ is a Fuchsian group, the stabiliser of any point $z\in \mathbb H$ is finite and cyclic. Denote by $stab_i$ the stabiliser of  a point in an orbit $O_i$ (defined up to conjugacy).

\begin{theorem}\label{fixed} If $\Gamma$ is any Fuchsian group and $O_1,O_2,\cdots, O_n$ are disjoint orbits in $\mathbb D$ of $\Gamma$. 
Then there is a non-zero function in $A^2_{s-2}$ with a zero of  order at least $v_i$ on all points of $O_i$ iff $$s>1+\frac{4\pi}{covolume(\Gamma)}\sum_i \frac{v_i}{|stab_i|}$$
\end{theorem}
\begin{proof} ($\implies$)Choose as in \ref{ultimate} an orderable subgroup
$\Psi<\Gamma$ with $n=[\Gamma:\Psi]<\infty$. The action of the stabiliser of a point in $\mathbb H$ on $\Gamma/\Psi$ is 
free since if $\gamma (\mu\Psi)=\mu\Psi$ then $\gamma\mu=\mu\psi$ for some $\psi\in \Psi$ so $\gamma$ is conjugate to an element
of $\Psi$, but  the stabiliser is of finite order and $\Psi$ is torsion free. The action of $\Gamma$ on  $O_i$
is the same as the action on $\Gamma/ stab_i$. But the spaces $(stab_i\backslash \Gamma)/ \Psi$  and
$stab_i\backslash( \Gamma/ \Psi)$ are the same so  there are 
$\displaystyle\frac{[\Gamma:\Psi]}{|stab_i|}$ disjoint orbits of $\Psi$ in $O_i$. Thinking of $A^2_{s-2}$ as
a representation of  $M=vN_\omega(\Psi)$, the  orbit $O_i$ thus contributes 
$\displaystyle v_i\frac{[\Gamma:\Psi]}{|stab_i|}$
mutually $\Psi$-orthogonal trace vectors for $M$ by \ref{wandering}. Thus 
$$dim_M(A^2_{s-2})=\frac{s-1}{4\pi}covolume (\Gamma)[\Gamma:\Psi]\geq [\Gamma:\Psi]\sum_i \frac{v_i}{|stab_i|}.$$
Moreover as before the function  in the statement of the theorem vanishing on the orbit is actually orthogonal to the 
$M$-linear span of the trace vectors so the inequality is strict.

($\impliedby$) For each $i$ choose $z_i\in O_i$ and let $\epsilon^j_i$ be vectors such that 
$$\langle f,\epsilon^j_i\rangle =f^{(j)}(z_i) \mbox{  for each } 0\leq j\leq v_{i-1}$$
If $\gamma_i$ generates the stabiliser of $z_i$ we can clearly arrange the cocycle $\omega$ of the projective
repersentation $\pi_s$  so that $u_i^{|stab_i|}=1$,
$u_i$ being $\pi_s(\gamma_i)$. Moreover it is clear that $u_i\epsilon^j_i$ is a multiple of $\epsilon^j_i$, necessarily
by an $nth.$ root of unity so that the $u_i\epsilon^j_i$ are in eigenspaces of the$u_i$. Hence they are in the
image of projections in $\vgm$ of trace $\frac{1}{|stab_i|}$. Hence by (\romannumeral10) of \ref{couplingproperties}
we have $$dim_{{\vgm}}(\overline{{\vgm}\epsilon^j_i})\leq\frac{1}{|stab_i|}$$
Since the von Neumann dimension is subadditive, summing over $i$ and $j$ we get
$$\sum_{i,j}dim_{{\vgm}}(\overline{{\vgm}\epsilon^j_i})\leq \sum_i \frac{v_i}{|stab_i|}$$
which by hypothesis is less than $dim_{\vgm }A^2_{s-2}$.  

So there is a function $\xi\in A^2_{s-2}$ which is orthogonal to all the $\pi_s(\gamma)\epsilon^j_i$.
This means that $\xi$ vanishes to order at least $v_i$  on each $O_i$.
\end{proof}
We can now extend the calculation in \ref{density} of the density $D^+$ to all orbits of all Fuchsian groups.

\begin{corollary}  If $\Gamma$, a Fuchsian group,  and $z\in\mathbb D$ are given with the stabilizer of $z$ having
order $stab$ then 
 $$D^+(\Gamma(z))=\frac{2\pi}{stab\times\mbox{covolume }(\Gamma)}$$
\end{corollary}
\begin{proof} The proof is as in \ref{density}.
\end{proof}
Note that the result extends to more than one orbit, and if there were a density calculation for sets with
zeros of prescribed order, that density could be calculated for Fuchsian groups.

The following result is surely known to experts.
\begin{corollary} Let $f$ be a holomorphic $k$-differential on a Riemann surface $\mathbb D/\Gamma$ of genus $g$, lifted to give
a holomorphic function on $\mathbb D$.  Then $f$ is square integrable for the measure $\displaystyle (1-r^2)^{s-2}rdrd\theta$ for
every $s>2k+1$ but not for $s=2k+1$
\end{corollary}
\begin{proof} With our definition of cusp form, $f$ is a cusp form of weight $2k$ so the multiplication operator $M_f$ is by 
\ref{bounded} a bounded operator from $A^2_{-1+\epsilon}$ to $A^2_{-1+2k+\epsilon}$ for every $\epsilon>0$. The 
constant function $1$ is in $A^2_{-1+\epsilon}$ so $f$ itself is in $A^2_{-1+2k+\epsilon}$.
On the other the degree of the $k$th power of the canonical bundle is $2k(g-1)$ so by Riemann Roch $f$ has $2k(g-1)$
zeros counted with multiplicity. So by \ref{fixed} since $\Gamma$ acts freely we must have
 $s$ strictly greater than $\displaystyle 1+2k(g-1)\frac{4\pi}{covolume}=1+2k$.
\end{proof}

  \section{Trace vectors for the commutant of $\Gamma$.}
  
  We need an elementary result on Poincar\'e series, going back to Poincar\'e-\cite{poincare}. We prove it here because it is
  usually stated for $s$ a positive integer whereas we need it for real $s>1$ \cite{bers}. If $s$ is a positive even integer the
  next step after convergence is usually to show that the Poincar\'e series defines a cusp form. But for
  real $s$ this will not be the case because of the non-homomorphic nature of the branch of the logarithm.
  
  In the next lemma $F$ will be a fundamental domain for $\Gamma$, a Fuchsian group as usual with the convention established above for the meaning of $cz+d$,
  $s$ will be a real number bigger than one and a fixed branch of $log$ is used to define $(cz+d)^s$.
  
  \begin{lemma} \label{poincare}Let $\xi\in A^2_{s-2}$. Then the Poincar\'e series 
  $\displaystyle \sum_{\gamma \in \Gamma} \frac{\xi(\gamma(z))^2}{(cz+d)^{2s}}$
  converges locally uniformly in $\mathbb H$ as does $\displaystyle\sum_{\gamma \in \Gamma} \frac{|\xi(\gamma(z))|^2}{|cz+d|^{2s}}$, the former to a holomorphic function and the latter to (at least) a continuous one.
  \end{lemma}
  \begin{proof} Fix a ball $K$ in  $F$. Putting $\displaystyle f_\gamma(z)=\frac{\xi(\gamma(z))^2}{(cz+d)^{2s}}$, the square of the
  $L^2$ norm of $\xi$ is 
  $$\int_{\mathbb H} |\xi(z)|^2y^s \frac{dxdy}{y^2}=\sum_{\gamma\in \Gamma} \displaystyle \int_{F}
|f_\gamma(z)|y^s \frac{dxdy}{y^2}$$
Since $f_\gamma$ is holomorphic, by the mean value property there is a $C$ such that $\displaystyle |f_\gamma(z)|\leq  C||f_\gamma(z)||_1$ for
all $\gamma\in \Gamma$ and $z\in K$ where by $||-||_1$ we mean the $1$-norm on the fundamental domain $F$ for the measure $\displaystyle{\frac{dxdy}{y^{2-s}}}$. Since $\sum_\gamma  ||f||_1$ converges, convergence on $F$ of the two functions in the 
statement of the theorem  is guaranteed 
by the Weierstrass M-test. Locally uniform convergence everywhere follows by varying the fundamental domain.
  \end{proof}

  Radulescu in \cite{radulescu1},\cite{radulescu2} has given a description of the commutant $M=\vgm'$ on $A^2_{s-2}$. Given an $L^\infty$ function
 $f$ on $\mathbb H$ that is fixed by the action of $\Gamma$ (which is the same thing as an $L^\infty$ function
 on a fundamental domain), one can define the ``Toeplitz'' operator $T_f$ which is 
 the composition  $$T_f=PM_f:A^2_{s-2}\rightarrow A^2_{s-2}$$ where $P$ is the orthogonal projection from
 $L^2(\mathbb H)$ onto  $A^2_{s-2}$.
In \cite{radulescu2} we find:
\begin{theorem} The subspace of $M$ spanned by the $T_f$ is dense in the 2-norm $||x||_2=\sqrt{tr(x^*x)}$.
\end{theorem}
and 
\begin{theorem} \label{radtrace}The trace in $M$ of $T_f$ is a multiple of  $\displaystyle\int_F f(z)\frac{dxdy}{y^2}$.
\end{theorem}
Note that by $\Gamma$-invariance  the integral does not depend on the fundamental domain.
\begin{definition} \label{tracelike}An element $\xi\in A^2_{s-2}$ will be called \emph{tracelike} if, for all $z\in\mathbb H$,
 $$  \sum_{\gamma\in\Gamma}\frac{|\xi(z)|^2}{|cz+d|^{2s}}=(constant)Im(z)^{-s} $$

\end{definition}
\begin{corollary}\label{commutrace} A function $\xi\in A^2_{s-2}$ is a trace vector for $M$ iff 
it is tracelike.
\end{corollary}
\begin{proof} By \ref{radtrace} we have, up to constants, for a bounded $\Gamma$ invariant function on $\mathbb H$,
 $$\langle T_f\xi,\xi\rangle=\int_Ff(z) \frac{dxdy}{y^2}.$$
But $\displaystyle \langle T_f\xi,\xi\rangle=\int_{\mathbb H} f(z)|\xi(z)|^2y^s\frac{dxdy}{y^2}=\int_F f(z) \sum_{\gamma\in\Gamma}\frac{|\xi(\gamma(z))|^2}{|cz+d|^{2s}} y^s\frac{dxdy}{y^2}$. The series converges to a continuous
function by \ref{poincare}. When we subtract a constant times $y^{-s}$ we get a function orthogonal on $F$ to all bounded measurable functions. The corollary follows  by varying the fundamental domain. 
\end{proof}
\begin{remark} \rm{Cusp forms give us a supply of interesting Toeplitz operators.  We have seen in 
\ref{cuspbounded} that a cusp form $f$
gives a bounded $\vgm$-linear map $T_f$ between Bergman spaces. So if $f$ and $g$ are cusp forms of the
same weight $p$, $T_f^*T_g$ is in $M$. It is actually the Toeplitz operator for the $\Gamma$-invariant bounded
function $h(z)=\bar {f(z)} g(z) Im(z)^{p}$. Theorem \ref{radtrace} then shows that the trace in $M$ of
$T_f^*T_g$ is the integral of $h$ over the fundamental domain with hyperbolic measure, i.e. the well known Petersson inner product \cite{petersson}.
This result was also obtained in \cite{ghj}. Radulescu also claims in \cite{radulescu2} that the Toeplitz operators given by
cusp forms are dense in $M$ though it appears he has only proved it for \pslz.
}
\end{remark}
This allows us to state a theorem about existence of such functions.
\begin{theorem} There is a  tracelike $\xi\in A^2_{s-2}$  iff $\displaystyle s\leq1+\frac{4\pi}
{covolume(\Gamma)}$. Moreover if $\displaystyle s=1+\frac{4\pi}
{covolume(\Gamma)}$, the condition is \emph{equivalent} to $\xi$ being a wandering vector for $\Gamma$.
\end{theorem}

\begin{proof} These are immediate consequences of von Neumann dimension. Item (ix) of \ref{couplingproperties}
proves the first assertion and the equivalence of being a trace vector for $M$ and $M'$ is easy when 
the von Neumann dimension is one since then the Hilbert space is $M$-isomorphic to the $L^2$ closure of $M$.
\end{proof}
  %  This result has the corollary, amusing because its proof does not involve any contour integrals ([]):
  
  %\begin{corollary} The weight of a cusp form $T_f is at least $\frac{4\pi}{covolume(\Gamma)}$.
  %\end{corollary}
%\begin{proof}
%If $\varepsilon$ is small consider $T_f$ as a map from $A^2_{-1+\varepsilon} $ to $A^2_{-1+\varepsilon+p} $.

%\end{proof} 
 \appendix\label{linebundleuse}
\section {Existence of nontrivial central extensions of Fuchsian groups arising from the nonintegral values of $s$.}\label{obstruction}

We observed in section \ref{fuchsian}  that, for $s>1$, the formula $$(\check\pi_s(g^{-1})f)(z)=\frac{1}{(cz+d)^{s}}f(g(z))$$
only defines a \emph{projective} unitary representation of \pslr on $A^2_{s-2}$, but that on restriction to a Fuchsian
group $\Gamma$ the representation might be adjusted to become honest. That is the case for instance if $\Gamma$
is a free product of cyclic groups (\cite{HKS})- simply adjust the unitaries representing the generators so that they
have the right order in the unitary group of $A^2_{s-2}$. It would have simplified the presentation in this paper if we could
do the same for all Fuchsian groups, but in this appendix we will show that this is not the case for fundamental groups
of surfaces of genus bigger than one.

\iffalse Let $G$ denote \pslr and $\tilde G$ it's universal cover so there is an exact sequence
 $1\rightarrow \mathbb Z\rightarrow \tilde G\overset{\rho}\rightarrow G\rightarrow 1$.

\begin{lemma} For each real $s>1$ there is a unitary representation $\tilde \pi_s:\tilde G\rightarrow U(A^2_{s-2})$
so that $$(\tilde\pi_s(g^{-1})f)(z)=\frac{1}{(cz+d)^{s}}f(g(z))\quad \mbox{ if } \quad \rho(g)^{-1} =\begin{pmatrix}
a&b\\c&d\end{pmatrix}$$
where $(cz+d)^s=e^{s\log(cz+d)}$ for some branch of $\log$ depending on $g$.
\end{lemma}
\begin{proof} This amounts to showing that there is a branch of log so that $\displaystyle \frac{1}{(cz+d)^{s}}$ is a cocycle for $\tilde G$. Use the
universal cover property for $\tilde G$ to get a map to $\mathbb C$ lifting the exponential map to 
$\mathbb C\setminus \{0\}$. Check the cocycle property by checking the image under exp and continuity.
\end{proof}
\fi

\begin{theorem} \label{nolift}Let $\Gamma$ be the Fuchsian group of the fundamental group of a Riemann surface
$\Sigma$ of genus $g$.  The projective representation given by the restriction of $\check \pi_s$ (for $s>1$) is equivalent to an honest representation iff $\displaystyle s\in \frac{1}{g-1}\mathbb Z$.

\end{theorem}

\begin{proof} 
%Notation $p(\tilde\gamma)=\gamma \in \Gamma$. 
Let $\{\gamma_i\}$ be generators for
$\Gamma$ so that the defining relation for  $\Gamma$ is $\prod [\gamma_i,\gamma_{i+1}]=1$ (see \cite{fulton})
 then changing  liftings $\check \pi_s(\gamma)$ of the projective representation of $\Gamma$ to different unitaries does not affect $\prod [\check \pi(\gamma_i),\check \pi(\gamma_{i+1})]=1$ provided the liftings of inverses in $\Gamma$
 are inverse unitaries. By \ref{choiceofbranch} this is true for our careful definition of $\check \pi$.
 So the single number $$obstr(s)=\prod [\check \pi(\gamma_i),\check \pi(\gamma_{i+1})]$$ is exactly the
 obstruction to lifting the restriction of $\check \pi$ to an honest unitary representation.
 % So if $obstr(s)=1$  
 %we find that the map $\chi_\gamma(z)=(cz+d)^s$ is a cocyle (since the projective representation is honest), i.e. $$\chi_{\alpha\beta}(z)=\chi_\alpha(\beta(z))\chi_\beta(z)$$
 Observe also that $s\mapsto obstr(s)$, when written out as an explicit function of $s$ and $z$, is a continous homomorphism from 
 $\mathbb R$ to the circle $\mathbb T^1$.

 Thus the problem becomes: "what is the kernel of $obstr$?"
 (Any even integer $s$ is in the kernel since then $\displaystyle\frac{1}{(cz+d)^{s}}$ has the 
 cocycle property for all of \pslr.)This question has nothing to do with Bergman space.
 We will answer it using line bundles on $\Sigma$.
 We claim that  $ker(obstr)=\frac{1}{g-1}\mathbb Z$.
Our proof will use the following construction:
\begin{proposition}\label{linebundle}
Suppose $obstr(s)=1$ for some $s\in \mathbb R, s>1$. Then there is a holomorphic line bundle $L(s)$ over $\Sigma$ 
with the following two properties:
\begin{enumerate}
\item $L(s+t)=L(s)\otimes L(t) $
\item $L(2)$ is the canonical line bundle $K$.
\end{enumerate}

\end{proposition} 
\begin{proof}
Since $obstr(s)=1$, the projective representation of $\Gamma$ on $A^2_{s-2}$ may be lifted to an honest representation by defining,
 for  $w=\prod_1^n\alpha_i$where $\alpha_i$ is one of the generators for each $i=1,2,\cdots,n$
 $$\pi(w)=\prod_1^n\check \pi (\alpha_i)$$  and

It follows that $\pi({\gamma^{-1}})(\xi)(z)=j(\gamma,z)\xi(\gamma(z))$ where $j(\gamma)$ satisfies the cocycle
condition $j(\gamma_1\gamma_2,z)=j(\gamma_1,\gamma_2(z))j(\gamma_2,z)$, and $j(\gamma,z)$ is a 
product of holomorphic functions of $z$ of the form $exp(s\log(cz+d))$.

The cocycle condition is exactly what is required to define an action of $\Gamma$ on the line bundle (over 
 $\mathbb H$)
 $\mathbb H\times \mathbb C$:$$ \gamma(z,w)=(\gamma(z),j(\gamma,z) w)$$
 
 This action is properly discontinuous so, passing to the quotient, we obtain a line bundle $L(s)$ on $\Sigma$, which is
 holomorphic because $j$ is. 
 
 \begin{enumerate}
\item 
Change of local trivialisations of $L(s)$ are obtained by lifting to $\mathbb H\times \mathbb C$ and applying
elements of $\Gamma$, and tensor product of line bundles corresponds to multiplying the cocycle  defining the
action.Since $j$ is a product of terms $f(s)$ with $f(s+t)=f(s)f(t)$, the same is true for $j$ as a function of $s$.

\item When $s=2$, $\displaystyle\frac{1}{(cz+d)^{s}}$ is already a cocycle so it is equal to $j(\gamma, z)$. But the canonical
line bundle is that of holomorphic one-forms which are locally of the form $f(z)dz$ and transform under
the action of $\Gamma$ just as our action on $\mathbb H\times \mathbb C$ acts on functions. 
\end{enumerate}

\end{proof}

We return to the proof of \ref{nolift}. The rational number $2$ is in $ker(obstr)$ so it suffices to show:
\begin{enumerate}
\item $\displaystyle\frac{1}{g-1}\in ker(obstr)$

\item No rational number $r=1+\epsilon$, $0<\epsilon<\frac{1}{g-1}$ is in  $ker(obstr)$.
\end{enumerate} 
Let us begin with (2). Suppose $s=1+\frac{m}{n}\in ker(obstr)$ with $\frac{m}{n}<\frac{1}{g-1}$. Form the line bundle  $L(s)$ of \ref{linebundle} 
over $\Sigma$ and let its degree be $d$. Then $\otimes^{2n} L(s) \cong \otimes^{n+m}K$ by (2) of \ref{linebundle}.
Equating the degrees of both sides we get $2nd=2(g-1)(n+m)$ or $d=(g-1)(1+\frac{m}{n})$.
But $(g-1)\frac{m}{n}$ is not an integer, a contradiction.
\5
So we only have to show that $s=\frac{1}{g-1}\in ker(obstr)$. Note that another way of phrasing the lifting property
for $\check \pi_s$ is the following: does there exist a function $\omega:\Gamma\rightarrow \mathbb T$ so
that $\displaystyle \gamma \mapsto \frac{\omega(\gamma)}{(cz+d)^s}$ has the cocycle property.

Choose a holomorphic line bundle $L$ over $\Sigma$ of degree $1$. Tensoring $L$ if necessary by a line bundle
of degree $0$ we may assume that $\otimes^{2g-2} L$ is the canonical line bundle $K$ of holomorphic $1$-forms.
Now take the universal cover of $L$ to obtain $\tilde L$ over $\mathbb H$ which may be trivialised so that there
is an action of $\Gamma$ on $\mathbb H\times \mathbb C$ of the form 
$$\gamma(z,w)=(\gamma(z),\alpha(\gamma,z) w)$$
for some holomorphic cocycle $\alpha$. The $(2g-2)$th. power of $\alpha$ yields an action on
 $\mathbb H \times \mathbb C$ which is equivalent to the action yielding $K$, i.e. that coming from the
 cocycle $\displaystyle\frac{1}{(cz+d)^2}$. We conclude there is a nonvanishing holomorphic function $h(z)$ such that
 $$\alpha(\gamma,z)^{2g-2}=\frac{1}{(cz+d)^2}h(\gamma(z))h(z)^{-1}.$$
 Since $\mathbb H$ is simply connected choose for each $\gamma$ a branch of $\displaystyle h(\gamma(z))^{\frac{1}{2g-2}}$ to obtain
  $$\alpha(\gamma,z)=\omega(\gamma)\frac{1}{(cz+d)^s}h(\gamma(z))^{\frac{1}{2g-2}}(h(z)^{\frac{1}{2g-2}})^{-1}.$$
  for some $2(g-2)$th. roots of unity $\omega(\gamma)$.
 Since both $\alpha$ and $\displaystyle h(\gamma(z))^{\frac{1}{2g-2}}(h(z)^{\frac{1}{2g-2}})^{-1}$ are (holomorphic) cocycles,
 so is $\displaystyle\frac{\omega(\gamma)}{(cz+d)^s}$ which means $s\in ker(obstr)$.

\end{proof}
We would like to acknowledge a lot of help from Dan Freed with appendix A.%\ref{linebundleuse}.
%The last piece of the puzzle is thus to calculate $s>0$ which generates the kernel of the map $obstr$. It is certainly
%rue that $s\leq 1$ because of the existence of spin structures on $\Sigma$.

\iffalse Sketch. By Bargmann the representations $\check\pi_s$ give unitary representations $\tilde \pi_s$ of the universal cover of
\pslr.  There is no lifting of $\check\pi_s$ as an honest representation since it is irreducible and would then be on
the list of irreps. It isn't. Now by line bundle theory the representation of $\pi_1(\Sigma)$ given by a compact Riemann
surface, into \pslr, doesn't lift to the universal cover. So since $\tilde \pi_s$ is faithful there can be no lift of 
$\check \pi_s |_{\pi_1(\Sigma)}$ since it could be pulled back to lift to the universal cover.
\fi

\section{An amusette: calculation of the algebra of modular forms.}

We will restrict our attention to \pslz\2though the method of this section surely applies in great generality. All the
results about modular forms in this section are extremely well known and elementary (\cite{zagier},\cite{serre}) but our derivation of them is somewhat different! The method should be applicable to a Fuchsian group provided there is an analogue of $\Delta$
(\cite{gannontriangle}). Relations in the algebra can be checked using zero sets.

\begin{lemma}\label{nocusp<12}
If $f$ is a cusp form of weight $p$  for \pslz then then $p\geq 12$.
\end{lemma}
\begin{proof} Choose a $w$ which is not a fixed point for \pslz. Then $g(z)=j(z)-j(w)$ vanishes on the orbit of $w$. 
And $fg$ is a modular form of weight $p$. Choosing a small $\epsilon>0$, $|\Delta^\epsilon fg|y^{\frac{p+12\epsilon}{2}}$ is
invariant under \pslz \2and bounded on a fundamental domain. Which gives 
$$|\Delta^\epsilon fg(z)|\leq (constant) y^{-\frac{p+12\epsilon}{2}}$$ So that for $\xi \in A^2_{s-2}$, 
$\Delta^\epsilon fg\xi\in A^2_{s-2+p+12\epsilon}$ and it vanishes on the orbit of $w$ which contradicts theorem \ref{ultimate}
if $s$ is close to $1$ and $\epsilon$ is small.
\end{proof}

\begin{corollary} There is no non-zero modular form of weight $2$.
\end{corollary}
\begin{proof}
Suppose $f$ were such a modular form. If $f$ has a zero, then multiplying it by a small positive power of $\Delta$
times a vector $\xi$ as in lemma \ref{nocusp<12} gives an $A^2$ function vanishing on an orbit which contradicts
\ref{ultimate}.

If $f$ vanishes nowhere, one may form $\displaystyle\frac{\Delta}{f}$ which is a cusp form of weight less than $12$, disallowed
by \ref{nocusp<12}.
\end{proof}
Given the above and the Eisenstein series it is not hard to determine the whole algebra of modular forms. Uniqueness
of the modular form of weight 12 is given by dividing by $\Delta$ and the maximum modulus theorem, as usual \cite{zagier}.
For weights $p=4,6,8$ and $10$ just subtract the appropriate multiple of $G_{p\over 2}$ to obtain a cusp form
which must be zero by \ref{nocusp<12}. %For $p=12$ one may divide by $\Delta$ and use the maximum 
%modulus theorem to conclude that the space of modular forms of weight 12 is one dimensional.

It is now routine to obtain the whole algebra of modular forms since
multiplication by $\Delta$ is clearly an injection of modular forms of weight $p$ onto cusp forms of weight $p+12$ 
and subtracting the appropriate multiple of the Eisenstein series gives a cusp form.
We conclude that the algebra of cusp forms is a graded commutative algebra freely generated by $G_2$ in
degree $4$ and $G_3$ in degree $6$. (See \cite{serre}.)

\end{document}